\documentclass[12pt,twoside]{amsart}
\usepackage{amssymb,amsmath,amsthm, amscd, enumerate, mathrsfs}
\usepackage{graphicx, hhline}
\usepackage[all]{xy}
\usepackage{color}
\usepackage{hyperref}
\title{On inversion of adjunction}
\author{Osamu Fujino and Kenta Hashizume}
\date{2021/7/5, version 0.08}
\keywords{adjunction, inversion of adjunction, minimal model program}
\address{Osamu Fujino \\ Department of 
Mathematics, Graduate School of Science, 
Kyoto University, Kyoto 606-8502, Japan}
\email{fujino@math.kyoto-u.ac.jp}
\address{Kenta Hashizume \\ Graduate School of Mathematical Sciences, 
The University of Tokyo, 3-8-1 Komaba Meguro-ku Tokyo 153-8914, Japan}
\email{hkenta@ms.u-tokyo.ac.jp}

\DeclareMathOperator{\Supp}{Supp}
\DeclareMathOperator{\Spec}{Spec}

\DeclareMathOperator{\Nlc}{Nlc}
\DeclareMathOperator{\Exc}{Exc}
\newtheorem{thm}{Theorem}[section]
\newtheorem{lem}[thm]{Lemma}

\newtheorem*{claim}{Claim}

\theoremstyle{definition}
\newtheorem{step}{Step}
\newtheorem{defn}[thm]{Definition}
\newtheorem{rem}[thm]{Remark}
\newtheorem*{ack}{Acknowledgments}  
\makeatletter
    
    \@addtoreset{equation}{section}
\makeatother

\begin{document}

\maketitle 

\begin{abstract}
We first announce our recent result on adjunction and 
inversion of adjunction. 
Then we clarify the relationship between 
our inversion of 
adjunction and Hacon's inversion 
of adjunction for log canonical centers of arbitrary codimension. 
\end{abstract}

\section{Introduction}

In \cite{fujino-hashizume2}, we established the following 
adjunction and inversion of adjunction for 
log canonical centers of arbitrary codimension in full generality. 

\begin{thm}[Adjunction and Inversion of Adjunction, 
see \cite{fujino-hashizume2}]\label{f-thm1.1}
Let $X$ be a normal 
variety and let $\Delta$ be an effective $\mathbb R$-divisor 
on $X$ such that 
$K_X+\Delta$ is $\mathbb R$-Cartier. 
Let $W$ be a log canonical center of $(X, \Delta)$ 
and let $\nu\colon Z\to W$ be the normalization of $W$. 
Then there exist a b-potentially nef $\mathbb R$-b-divisor 
$\mathbf M$ and an $\mathbb R$-b-divisor $\mathbf B$ 
on $Z$ such that 
$\mathbf B_Z$ is effective with 
$$
\nu^*(K_X+\Delta)=\mathbf K_Z+\mathbf M_Z+\mathbf B_Z. 
$$
More precisely, there exists a projective birational morphism 
$p\colon Z'\to Z$ from a smooth quasi-projective variety 
$Z'$ such that 
\begin{itemize}
\item[(i)] $\mathbf M=\overline {\mathbf M_{Z'}}$ and 
$\mathbf M_{Z'}$ is a potentially nef $\mathbb R$-divisor on $Z'$, 
\item[(ii)] $\mathbf K+\mathbf B=\overline {\mathbf K_{Z'}+
\mathbf B_{Z'}}$, 
\item[(iii)] $\Supp \mathbf B_{Z'}$ is a simple 
normal crossing divisor on $Z'$, 
\item[(iv)] $\nu\circ p\left(\mathbf B^{>1}_{Z'}\right)=W\cap \Nlc(X, \Delta)$ 
holds set theoretically, and 
\item[(v)] $\nu\circ p\left(\mathbf B^{\geq 1}_{Z'}\right)=W
\cap \left(\Nlc(X, \Delta)\cup \bigcup _{W\not\subset W^\dag}W^\dag
\right)$, where $W^\dag$ runs over log canonical centers of $(X, 
\Delta)$ which do not contain $W$, holds set theoretically.  
\end{itemize}
\end{thm} 

For the details of Theorem \ref{f-thm1.1}, see \cite{fujino-hashizume2}. 
On the other hand, Hacon introduced a b-divisor, which is 
denoted by $\mathbf B(W; X, \Delta)$, and formulated 
his inversion of adjunction for log canonical centers of 
arbitrary codimension in \cite{hacon}. 
We note that the definition of $\mathbf B(W; X, \Delta)$ is different 
from our definition of $\mathbf B$ in Theorem \ref{f-thm1.1}. 
The goal of this paper is to prove that 
$\mathbf B(W; X, \Delta)=\mathbf{B}$ always holds.  
The following theorem is the main result of this paper. 

\begin{thm}\label{f-thm1.2}
Let $X$ be a normal variety and let $\Delta$ be an 
effective $\mathbb R$-divisor on $X$ such that 
$K_X+\Delta$ is $\mathbb R$-Cartier. Let 
$W$ be a log canonical center of $(X, \Delta)$. Then 
Hacon's $\mathbf B(W; X, \Delta)$ coincides with the 
$\mathbb R$-b-divisor $\mathbf B$ on $Z$ in 
Theorem \ref{f-thm1.1}, where 
$Z$ is the normalization of $W$. 
Hence our adjunction and inversion of adjunction for 
log canonical centers of arbitrary codimension 
completely generalizes Hacon's inversion of adjunction. 
\end{thm}

By Theorem \ref{f-thm1.2}, Hacon's inversion of 
adjunction for log canonical centers of arbitrary 
codimension in \cite{hacon} 
now becomes a very special case of 
Theorem \ref{f-thm1.1}. 
We think that the definition of $\mathbf B$ in \cite{fujino-hashizume1} 
and \cite{fujino-hashizume2} is 
more natural than Hacon's definition of $\mathbf B(W; X, \Delta)$ in 
\cite{hacon}. 
However, $\mathbf B(W; X, \Delta)$ seems to be 
easier to compute than $\mathbf B$. 
Hence Theorem \ref{f-thm1.2} is important and useful. 

\begin{rem}\label{f-rem1.3}
In \cite{hacon}, 
Hacon defined $\mathbf B(W; X, \Delta)$ under the extra assumption 
that $\Delta$ is a boundary $\mathbb Q$-divisor on $X$ 
such that $K_X+\Delta$ is $\mathbb Q$-Cartier. 
However, his definition works for 
effective $\mathbb R$-divisors $\Delta$ such that 
$K_X+\Delta$ is $\mathbb R$-Cartier 
without any modifications. 
By definition, it is obvious that $\mathbf B(W; X, \Delta)\leq \mathbf B$ 
always holds. 
\end{rem}

Let us quickly explain the proof of Theorem \ref{f-thm1.1} for 
the reader's convenience. First we take a suitable 
resolution of singularities of the pair $(X, \Delta)$. 
Next, by using the framework of quasi-log 
schemes (see \cite[Chapter 6]{fujino-foundations}), we construct 
a natural quasi-log scheme structure on $Z$ (see \cite{fujino-cone}). 
Then we apply the theory of basic $\mathbb R$-slc-trivial 
fibrations and obtain $\mathbf B$ and $\mathbf M$ 
satisfying (i), (ii), (iii), and (v) 
(see \cite{fujino-hashizume2}). 
Finally, we prove (iv) with the aid of the minimal model 
program for log canonical pairs (see \cite{fujino-hashizume1}). 
We strongly recommend the interested reader to see \cite{fujino-cone}, 
\cite{fujino-hashizume1}, and \cite{fujino-hashizume2}. 

In this paper, we will only use the minimal model program at 
the level of \cite{bchm}. 
We will freely use the standard notation and 
definitions of the minimal model 
program as in \cite{fujino-fundamental} and \cite{fujino-foundations} 
(see also \cite{fujino-cone}). 

\begin{ack}
The first author was partially 
supported by JSPS KAKENHI Grant Numbers 
JP19H01787, JP20H00111, 
JP21H00974. 
The second author was partially supported by JSPS KAKENHI Grant
Numbers JP16J05875, JP19J00046. 
\end{ack}

\section{$\mathbf B(W; X, \Delta)$ and $\mathbf B$}

Let us recall the definition of $\mathbf B(W; X, \Delta)$ and 
$\mathbf B$. 

\begin{defn}[$\mathbf B(W; X, \Delta)$ and $\mathbf B$]
\label{f-def2.1}
Let $X$ be a normal 
variety and let $\Delta$ be an $\mathbb R$-divisor 
on $X$ such that 
$K_X+\Delta$ is $\mathbb R$-Cartier and that 
$\Delta$ is effective in a neighborhood of the generic point of 
a closed subvariety $W$. 
Assume that $W$ is a log canonical center of $(X, \Delta)$. 
Let $\nu\colon Z\to W$ be the normalization of $W$. 
For any proper birational morphism $\rho\colon \tilde{Z} \to Z$ 
from a normal variety $\tilde Z$, 
we consider prime divisors $T$ over $X$ such 
that $a(T,X,\Delta)=-1$ and that the center of $T$ on $X$ is $W$. 
We take a suitable resolution $f\colon Y\to X$ 
with $K_Y+\Delta_Y=f^*(K_X+\Delta)$ so that 
$\Delta_Y$ is a simple normal crossing divisor on $Y$, 
$T$ is a prime divisor on $Y$, 
and the induced map $f_{T}\colon T\dashrightarrow \tilde{Z}$ is a morphism. 
We put $\Delta_{T}=(\Delta_{Y}-T)|_{T}$. 
For any prime divisor $P$ on $\tilde{Z}$, 
we shrink $\tilde Z$ and assume that $P$ is Cartier. 
Then we define a real number 
$\alpha_{P,T}$ by
$$\alpha_{P,T}
=\sup\{\lambda \in \mathbb R\,|\,(T,\Delta_T+\lambda f^*_TP) 
\ \text{is sub log canonical over the generic point of}\ P\}.$$ 
It is easy to see that $\alpha_{P, T}$ is independent 
of the resolution $f\colon Y\to X$ and is well-defined. 
The trace $\mathbf B_{\tilde{Z}}$ of $\mathbf B$ on $\tilde{Z}$ is defined by
$$\mathbf B_{\tilde{Z}}=\sum_{P}(1-\underset{T}\inf \alpha_{P,T})P$$
where $P$ runs over prime divisors on $\tilde Z$ and 
$T$ 
runs over prime divisors 
over $X$ such that $a(T,X,\Delta)=-1$ 
and that the center of $T$ on $X$ is $W$. 

We choose and fix 
one prime divisor $T$ over $X$ such that 
$a(T, X, \Delta)=-1$ and that the center of $T$ on $X$ is $W$. 
The trace $\mathbf B(W; X, \Delta)_{\tilde Z}$ 
of $\mathbf B(W; X, \Delta)$ is defined by 
$$
\mathbf B(W; X, \Delta)_{\tilde Z}=\sum _P (1-\alpha _{P, T})P
$$ 
where $P$ runs over prime divisors on $\tilde Z$. 
By definition, $\mathbf B(W; X, \Delta)\leq \mathbf B$ always 
holds. 
\end{defn}

\begin{rem}\label{f-rem2.2}
Although it is not obvious, we can check that 
$\mathbf B$ is a well-defined $\mathbb R$-b-divisor 
on $Z$. On the other hand, 
we can easily see that 
$\mathbf B (W; X, \Delta)$ is 
a well-defined $\mathbb R$-b-divisor 
on $Z$, but it is not clear whether 
$\mathbf B(W; X, \Delta)$ is 
independent of the choice of $T$ or not. 
In Theorem \ref{f-thm1.2}, we prove 
that $\mathbf B=\mathbf B(W; X, \Delta)$ holds. 
This implies that $\mathbf B(W; X, \Delta)$ is 
independent of the choice of $T$. 
Moreover, by the proof of Theorem \ref{f-thm1.2}, 
the well-definedness of $\mathbf B$ is clear. 
\end{rem}

Precisely speaking, Hacon claims that $\mathbf B(W; X, \Delta)$ 
is independent of the choice of the divisor $T$ without proof 
in \cite{hacon}. 
In this paper, we prove it in a slightly more general 
setting. 

\section{Proof of Theorem \ref{f-thm1.2}} 

In this section, we prove Theorem \ref{f-thm1.2}. 
Before the proof of Theorem \ref{f-thm1.2}, we prepare three lemmas.

\begin{lem}\label{f-lem3.1} 
Let $X$ be a normal variety and let $\Delta$ be an $\mathbb R$-divisor 
on $X$ such that $K_X+\Delta$ is 
$\mathbb R$-Cartier. Let $f\colon X\to Y$ be a projective 
surjective morphism onto a smooth curve $Y$ such that 
$K_X+\Delta\sim _{\mathbb R, f}0$. 
Let $P$ be a closed point of $Y$ such that 
$(X, \Delta)$ is divisorial log terminal over $Y\setminus P$. 
We take the log canonical threshold $b_P$ of $(X, \Delta)$ with respect to 
$f^*P$. 
Let $F$ be a connected component of $f^{-1}(P)$. 
Assume that $F$ contains a log canonical center of $(X, \Delta+b_Pf^*P)$. 
Let $S$ be an irreducible component of $\left( \Delta^h\right)^{=1}$, 
that is, $S$ is a codimension one 
log canonical center of $(X, \Delta)$ which is dominant 
onto $Y$ by $f$. 
Then $S\cap F$ always contains a log canonical 
center of $(X, \Delta+b_Pf^*P)$. 
Hence, if $\nu\colon S^\nu\to S$ is the normalization 
of $S$ and $\Delta_{S^\nu}$ is the 
$\mathbb R$-divisor on $S^\nu$ defined by 
$K_{S^\nu}+\Delta_{S^\nu}=\nu^*(K_X+\Delta)$, 
then $(S^{\nu}, \Delta_{S^{\nu}}+b_P\nu^{*}(f|_{S})^{*}P)$ 
has a log canonical center mapping to $P$.
\end{lem}

\begin{proof} 
Without loss of generality, we may assume that 
$Y$ is quasi-projective by shrinking $Y$ around $P$. 
By replacing $\Delta$ with $\Delta+mf^*P$ for 
some sufficiently large positive integer $m$, 
we may assume that $\Delta$ is effective 
with $\Delta\geq f^*P$. 
In this situation, the log canonical threshold 
$b_{P}$ is a nonpositive number. 
We take a resolution of singularities of $X$ suitably and run a 
minimal model program with scaling of an ample divisor 
as in the proof of \cite[Theorem 3.9]{fujino-cone}. 
Then we have 
a dlt blow-up 
$g\colon Z\to X$ of $(X, \Delta)$ 
with $K_Z+\Delta_Z=g^*(K_X+\Delta)$ such that 
\begin{itemize}
\item $Z$ is $\mathbb Q$-factorial, 
\item $g$ is small over $Y\setminus P$, and 
\item the pair $\left(Z, {\Delta_Z^{<1}+\Supp \Delta_Z^{\geq 1}}\right)$ is 
divisorial log terminal. 
\end{itemize} 
For the details of the construction of $g\colon Z\to X$, 
see \cite[Theorem 3.9]{fujino-cone}. 
By replacing $f\colon (X, \Delta)\to Y$ and $F$ with 
$f\circ g\colon (Z, \Delta_Z)\to Y$ and $g^{-1}(F)$ respectively, 
we may assume that $X$ is $\mathbb Q$-factorial 
and $\left(X, {\Delta^{<1}+\Supp \Delta^{\geq 1}}\right)$ 
is divisorial log terminal. 
Hence $\left(X, (\Delta+b_Pf^*P)^{>0}\right)$ is a 
$\mathbb Q$-factorial divisorial log terminal pair. 
Note that 
$$K_X+(\Delta+b_Pf^*P)^{>0}
\sim _{\mathbb R, f}-(\Delta+b_Pf^*P)^{<0}\geq 0. 
$$ 
If $\Supp (\Delta+b_Pf^*P)^{=1}\supset F$, then it is obvious 
that $\Delta+b_Pf^*P$ is effective in a neighborhood of 
$F$ and that $S\cap F$ contains a log canonical center of $(X, \Delta+b_Pf^*P)$. 
Therefore, from now on, we assume that 
$\Supp (\Delta+b_Pf^*P)^{=1}\not \supset F$. 
Let $F=\sum _i F_i$ be the irreducible decomposition of $F$. 
If $F_i\not \subset \Supp (\Delta+b_Pf^*P)^{=1}$, 
then we take $0<\varepsilon _i\ll 1$. 
If $F_i\subset \Supp (\Delta+b_P f^*P)^{=1}$, 
then we put $\varepsilon _i=0$. 
Let $f\colon X\overset{h}{\longrightarrow} \bar Y\to Y$ 
be the Stein factorization.
Then $\left(X, (\Delta+b_Pf^*P)^{>0}
+\sum _i \varepsilon _i F_i\right)$ is a $\mathbb Q$-factorial divisorial 
log terminal pair with 
$$K_X+(\Delta+b_Pf^*P)^{>0}+\sum _i \varepsilon _i F_i
\sim _{\mathbb R, h} -(\Delta+b_Pf^*P)^{<0} 
+\sum _i \varepsilon _i F_i\geq 0. 
$$ 
We shrink $\bar Y$ around $h(F)$ and run 
a minimal model program of $K_X+(\Delta+b_Pf^*P)^{>0}+
\sum _i \varepsilon _i F_i$ over $\bar Y$ with scaling of an ample divisor. 
Then, after finitely many steps, 
we get $X'$ with the following commutative diagram: 
$$
\xymatrix{
X\ar[rd]_-h \ar@{-->}[rr]^-{\phi}&& X'\ar[ld]^-{h'} \\ 
& \bar Y& 
}
$$ 
such that $\Delta'+b_P(f')^*P$ is effective in a neighborhood of 
$F'$ and that 
$(\Delta'+b_P(f')^*P)^{=1}\geq F'$, where 
$f'\colon X'\to \bar Y$, $\Delta'=\phi_*\Delta$, and $F'=\phi_*F$. 
For the details of the above minimal model program, 
see \cite{fujino-semistable} (see also the techniques of very exceptional 
divisors discussed in \cite[Section 3]{birkar-flip}). 
This means that $S'\cap F'$ contains a log canonical center 
of $(X', \Delta'+b_P(f')^*P)$, where $S'=\phi_*S$ as usual. 
Hence there exists a prime divisor $E$ over 
$S'$ such that $a(E, S', \Delta_{S'}+b_P(f'|_{S'})^*P)=-1$ and $E$ maps to $P$, 
where $K_{S'}+\Delta_{S'}=(K_{X'}+\Delta')|_{S'}$. 
By the construction of $\phi\colon X\dashrightarrow X'$, we have 
$a(E,S, \Delta_{S}+b_P(f|_{S})^*P)=-1$. 
We note that $(X, \Delta^{<1}+\Supp \Delta^{\geq 1})$ is divisorial 
log terminal. 
Hence $S\cap F$ always contains a log canonical center of $(X, \Delta 
+b_Pf^*P)$.
\end{proof}

\begin{rem}\label{f-rem3.2}
In Lemma \ref{f-lem3.1}, we assume that $F$ contains no log canonical 
center of $(X, \Delta+b_Pf^*P)$. Then $(X, \Delta+(b_P+\varepsilon)f^*P)$ 
is log canonical in a neighborhood of $F$ for $0<\varepsilon \ll 1$. 
In this situation, $(S^\nu, \Delta_{S^\nu}+(b_P+\varepsilon) 
\nu^*(f|_S)^*P)$ is log canonical by adjunction. 
Hence $(S^\nu, \Delta_{S^\nu}+b_P 
\nu^*(f|_S)^*P)$ has no log canonical center 
mapping to $P$. 
\end{rem}

\begin{lem}\label{f-lem3.3}
Let $X$ be a normal quasi-projective variety and let $W$ be a 
closed subvariety 
of $X$. 
Let $\varphi\colon \widetilde W\to W$ be a projective birational morphism 
from a normal variety $\widetilde W$. 
Then we can construct a projective birational morphism 
$\psi\colon \widetilde X\to X$ from a normal 
variety $\widetilde X$ such that $\psi$ is an isomorphism 
over the generic point of $W$ with 
the following commutative diagram: 
\begin{equation}\label{f-eq3.1}
\xymatrix{
\widetilde W \ar@{^{(}->}^-{\widetilde
\iota}[d]\ar[r]^-{\varphi} & W\ar@{^{(}->}^-{\iota}[d] \\
\widetilde X \ar[r]_-{\psi}& X, 
}
\end{equation}
where $\iota$ and $\widetilde \iota$ are closed embeddings. 
\end{lem}

\begin{proof}
By \cite[Chapter II, Theorem 7.17]{hartshorne}, 
there exists a coherent ideal sheaf $\mathcal I$ on $W$ such that 
$\varphi\colon \widetilde W\to W$ corresponds to the blow-up 
of $\mathcal I$. 
We put 
$\mathcal J=
\mathrm{Ker} (\mathcal O_X\to \mathcal O_W\to \mathcal O_W/\mathcal I)$. 
Then $\mathcal J$ is a coherent ideal sheaf on $X$. 
Let $\psi'\colon X'\to X$ be the blow-up of $\mathcal J$. 
Then we obtain the following commutative diagram 
by \cite[Chapter II, Corollary 7.15]{hartshorne}. 
$$
\xymatrix{
\widetilde W \ar@{^{(}->}^-{\iota'}[d]
\ar[r]^-{\varphi} & W\ar@{^{(}->}^-{\iota}[d] \\
X' \ar[r]_-{\psi'}& X 
}
$$ 
By construction, $\psi'$ is an isomorphism 
over the generic point of $W$. 
Let $\nu\colon \widetilde X\to X'$ be the normalization of $X'$. 
Since $\widetilde W$ is normal by assumption, 
$\widetilde W\to X'$ factors through $\widetilde X$. 
Then we get the desired diagram \eqref{f-eq3.1} such that 
$\widetilde W\to \widetilde X$ is a closed embedding and that 
$\psi\colon \widetilde X\to X$ is an isomorphism 
over the generic point of $W$. 
\end{proof}

\begin{lem}\label{f-lem3.4}
Let $(Y, \Delta_Y)$ be a $\mathbb Q$-factorial divisorial log terminal 
pair and let $f\colon Y\to X$ be a projective 
morphism with $f_*\mathcal O_Y\simeq \mathcal O_X$ and $K_Y+\Delta_Y\sim 
_{\mathbb R, f}0$. 
Let $V$ be a reduced divisor on $Y$ such that 
$V\leq \Delta^{=1}_Y$ and $f(V_i)$ is independent of $i\in I$, 
where $V=\sum _{i\in I}V_i$ is the irreducible decomposition of $V$. 
We set $W=f(V_{i})$.
Suppose that no log canonical center of $(Y, \Delta_Y-V)$ maps to 
$W$ by $f$. 
Then $f_*\mathcal O_V\simeq \mathcal O_W$ holds. 
\end{lem}

\begin{proof}
We can take a projective birational morphism 
$g\colon Z\to Y$ from a smooth variety $Z$ such that 
$g$ is an isomorphism over the generic point of any 
log canonical center of $(Y, \Delta_Y)$ and 
that $\Exc(g)\cup \Supp g^{-1}_*\Delta_Y$ is a simple normal 
crossing divisor on $Z$. 
Then we can write 
$
K_Z+\Delta_Z=g^*(K_Y+\Delta_Y)+E
$ 
with $\Delta_Z=g^{-1}_*\Delta_Y$ such that 
$\lceil E\rceil$ is effective and $g$-exceptional. 
Let $V_Z$ be the strict transform of $V$ on $Z$. 
We have $-V_Z+\lceil E\rceil\sim
_{\mathbb R, f\circ g}K_Z+\Delta_Z-V_Z+\{-E\}$. 
By the vanishing theorem of Reid--Fukuda type (see 
\cite[Theorem 3.2.11]{fujino-foundations}), 
$R^ig_*\mathcal O_Z(-V_Z+\lceil E\rceil)=0$ for every $i>0$. 
We note that $g_*\mathcal O_Z(-V_Z+\lceil E\rceil)\simeq 
\mathcal O_Y(-V)$ holds since $\lceil E\rceil$ is effective 
and $g$-exceptional. We consider 
the following long exact sequence: 
\begin{equation*}
0 \longrightarrow f_*\mathcal{O}_{Y}(- V) 
\longrightarrow f_*\mathcal{O}_{Y} \longrightarrow f_*\mathcal O_{V} 
\overset{\delta}{\longrightarrow} 
R^{1}f_*\mathcal{O}_{Y}(- V) 
\longrightarrow \cdots. 
\end{equation*}
By \cite[Theorem 6.3 (i)]{fujino-fundamental} 
(see also \cite[Theorem 3.16.3 (i)]{fujino-foundations}), 
there exists no associated prime of $R^1f_*\mathcal O_Y(-V)\simeq 
R^1(f\circ g)_*\mathcal O_Z(-V_Z+\lceil E\rceil)$ 
in $W=f(V)$. 
Hence the above connecting homomorphism $\delta$ is zero. 
Therefore, $\mathcal O_X\simeq f_*\mathcal O_Y\to f_*\mathcal O_V$ 
is surjective. Thus the natural map $\mathcal O_W\to f_*\mathcal O_V$ 
is an isomorphism. 
\end{proof}

Let us start the proof of Theorem \ref{f-thm1.2}. 

\begin{proof}[Proof of Theorem \ref{f-thm1.2}]
We will prove $\mathbf B=\mathbf B(W; X, \Delta)$ under 
a slightly weaker assumption that 
$\Delta$ is only effective in a neighborhood of the 
generic point of $W$. 
\setcounter{step}{0}
\begin{step}\label{f-step1}
We take an arbitrary projective birational morphism 
$\tilde Z\to W$ from a normal variety $\tilde Z$. 
We have to prove $\mathbf B_{\tilde Z}=\mathbf B(W; X, \Delta)_{\tilde Z}$. 
By taking an affine open cover of $X$, 
we may assume that $X$ is quasi-projective. 
By applying Lemma \ref{f-lem3.3} to $\tilde Z\to W$ and $W\hookrightarrow 
X$, 
we get the following commutative diagram: 
\begin{equation*}
\xymatrix{
\tilde Z \ar@{^{(}->}
[d]\ar[r]& W\ar@{^{(}->}[d] \\
\widetilde X \ar[r]_-{\psi}& X, 
}
\end{equation*}
where $\psi\colon \widetilde X\to X$ is a projective 
birational morphism from a normal variety such that 
$\psi$ is an isomorphism over the generic point of $W$. 
We put $K_{\widetilde X}+\widetilde \Delta=\psi^*(K_X+\Delta)$. 
Since $\psi$ is an isomorphism over the generic point of $W$, 
$\widetilde \Delta$ is effective in a neighborhood of the 
generic point of $\tilde Z$ and 
$\tilde Z$ is a log canonical center of $(\widetilde X, \widetilde \Delta)$. 
By replacing $(X, \Delta)$ and $W$ with 
$(\widetilde X, \widetilde \Delta)$ and $\tilde Z$, respectively, 
we may further assume that $W$ is normal. 
By this reduction, all we have to do is to prove 
$\mathbf B_W=\mathbf B(W; X, \Delta)_W$. 
\end{step} 
\begin{step}\label{f-step2}
We take an effective Cartier divisor $D$ on $X$ such that 
$W\not\subset \Supp D$ and $\Supp \Delta^{<0}\subset 
\Supp D$. 
We consider the pair $(X, \Delta+mD)$ for some sufficiently large positive 
integer $m$ such that $\Delta+mD$ is effective. 
We take a projective birational morphism $f\colon Y\to X$ from a 
smooth quasi-projective variety $Y$ such that 
$f^{-1}(W)$ and $\Exc(f)$ are divisors on $Y$ such that 
the union of $f^{-1}(W)$, $\Exc(f)$, $\Supp f^{-1}_*\Delta$, and $\Supp 
f^{-1}_*D$ is contained in a simple normal crossing divisor. 
We put $K_Y+\Delta_Y=f^*(K_X+\Delta)$ and 
$K_Y+\Delta_Y+mf^*D=f^*(K_X+\Delta+mD)$. 
We define $V=\sum _{i\in I}V_i$, where 
$V_i$ runs over components of $\Delta^{=1}_Y$ with $f(V_i)=W$. 
By construction, there exists no log canonical center 
of $(Y, \Delta_Y-V)$ mapping to $W$ over a 
neighborhood of the generic point of $W$. 
Let $T$ be the prime divisor over $X$ 
which was chosen in order to define $\mathbf B(W; X, \Delta)$. 
Let $S$ be any prime divisor over $X$ such that 
$a(S, X, \Delta)=-1$ and that the center of $S$ on $X$ is $W$. 
We may assume that $S$ and $T$ 
are components of $V$ by taking $f\colon Y\to X$ suitably. 
To prove Theorem \ref{f-thm1.2}, it is sufficient 
to check that $\alpha_{P, V_i}$ is independent of the choice of $i \in I$. 

\end{step}
\begin{step}\label{f-step3} 
By running a minimal model program 
with scaling of an ample divisor 
as in the proof of \cite[Theorem 3.9]{fujino-cone}, 
we get a dlt blow-up 
$f'\colon Y'\to X$ of $(X, \Delta+mD)$ with the following 
commutative diagram. 
$$
\xymatrix{
Y \ar[dr]_-{f}\ar@{-->}^-\phi[rr]&& Y'\ar[dl]^-{f'} \\
& X &
}
$$
We put $K_{Y'}+\Delta_{Y'}=(f')^*(K_X+\Delta)$. 
By construction, $Y'$ is $\mathbb Q$-factorial and 
$$\left(Y', \left(\Delta_{Y'}+(f')^*mD\right)^{<1}+\Supp 
\left(\Delta_{Y'}+(f')^*mD\right)^{\geq 1}\right)$$ is divisorial log terminal. 
For the details, see \cite[Theorem 3.9]{fujino-cone}. 
Therefore, $f'\colon (Y', \Delta_{Y'})\to (X, \Delta)$ 
is a dlt blow-up over a neighborhood of 
the generic point of $W$. 
By construction again, $\phi$ does not contract any components of $V$. 
We put $V'=\phi_*V$ and $V'_i=\phi_*V_i$ for every $i\in I$. 
We note that $(Y', V')$ is divisorial log terminal 
since $V'\leq \Supp\left(\Delta_{Y'}+(f')^*mD\right)^{\geq 1}$. 
In particular, $V'_i$ is normal for every $i\in I$. 
We can take a 
Zariski open neighborhood $U$ of the generic point of $W$ 
over which $f'\colon (Y', \Delta_{Y'})\to 
(X, \Delta)$ is a dlt blow-up and 
no log canonical center of $(Y', \Delta_{Y'}-V')$ maps to $W$ by $f'$. 
By applying Lemma \ref{f-lem3.4} to $f'\colon (Y', 
\Delta_{Y'})|_{(f')^{-1}(U)}\to (f')^{-1}(U)$, we obtain that 
the natural map 
$\mathcal O_W\to 
f'_*\mathcal O_{V'}$ 
is an isomorphism on $U$. 
On the other hand, 
since every irreducible component of $V'$ is 
dominant onto $W$, we see that 
$\Spec_{W}f'_*\mathcal O_{V'}$ is a variety. 
Therefore, since $W$ is normal, the finite birational morphism 
$\Spec_{W}f'_*\mathcal O_{V'}\to W$ is an isomorphism by Zariski's 
main theorem. 
This implies that $f'\colon V'\to W$ has connected fibers. 
We put $K_{V'_i}+\Delta_{V'_i}=(K_{Y'}+\Delta_{Y'})|_{V'_i}$ and 
$f'_i=f'|_{V'_i}\colon V'_i\to W$ for every $i\in I$. 
We shrink $W$ and assume that $P$ is Cartier. 
Then we set 
$$\alpha_{P, V'_i}
=\sup\{\lambda \in \mathbb R\,|\,(V'_i,\Delta_{V'_i}+\lambda (f'_i)^*P) 
\ \text{is sub log canonical over the generic point of}\ P\}.$$ 
It is easy to see that $\alpha_{P, V_i}=\alpha_{P, V'_i}$ holds 
for every $i\in I$. 
Therefore, 
to prove that $\alpha_{P, V_i}$ is independent 
of the choice of $i \in I$, it is sufficient to prove that 
$\alpha_{P, V'_i}$ is independent of the choice of $i \in I$. 
\end{step}

\begin{step}\label{f-step4} 
We take a prime divisor $P$ on $W$. 
By cutting down $X$ by general hyperplanes, 
we assume that $W$ is a smooth curve and $P$ is a closed 
point. 
By shrinking $X$ suitably around $P$, 
$(V'_i, \Delta_{V'_i})$ is divisorial log terminal 
over $W\setminus P$ for every $i\in I$. 
We put 
$$
c_P=\sup \{\lambda \in \mathbb R\, |\, (V'_i, \Delta_{V'_i}+
\lambda (f'_i)^*P) \ 
\text{is sub log canonical for every $i\in I$}\}. 
$$ 
By definition, there exists $i_0\in I$ such that $\alpha_{P, V'_{i_0}}=c_P$ holds. 
From now on, we will prove that 
$\alpha_{P, V'_i}=c_P$ holds for every $i\in I$. 
If $\# I=1$, then there is nothing to prove. 
Hence we may assume that $\# I\geq 2$. 
To obtain $\alpha_{P, V'_i}=c_P$ for every $i\in I$, it is sufficient 
to prove the following claim. 
\begin{claim}\label{f-claim}
Let $F$ be any connected component of $f^{-1}_i(P)$ and 
let $B$ be any irreducible component of $(\Delta^h_{V'_i})^{=1}$ for 
some $i\in I$. Then $B\cap F$ contains a log canonical 
center of $(V'_i, \Delta_{V'_i}+c_P(f'_i)^*P)$. 
\end{claim} 
\begin{proof}[Proof of Claim]
We note that for any $j\in I$ there exists 
some $k\in I$ with $k\ne j$ such that 
$V'_j\cap V'_k\ne \emptyset$ and that 
some irreducible component of $V'_j\cap V'_k$ is dominant 
onto $W$ by $f'$ since $f'\colon V'\to W$ has connected 
fibers and every irreducible component 
of $V'$ is dominant onto $W$ by $f'$. 
We take an irreducible component $A$ of $V'_j\cap V'_k$ with 
$j\ne k$ such that 
$A$ is dominant onto $W$ by $f'$. 
We note that $A$ is an irreducible component of 
$(\Delta^h_{V'_j})^{=1}$ and $(\Delta^h_{V'_k})^{=1}$ by adjunction. 
Let $G_j$ be a connected component of $(f'_j)^{-1}(P)$. 
Then $A\cap G_j$ contains a log canonical center of 
$(V'_j, \Delta_{V'_j}+c_P(f'_j)^*P)$ if and only if $A\cap G_k$ 
contains a log canonical center of 
$(V'_k, \Delta_{V'_k}+c_P(f'_k)^*P)$, where $G_k$ is the 
connected component of $(f'_k)^{-1}(P)$ containing $A\cap G_j$. 
We first apply Lemma \ref{f-lem3.1} to $f'_{i_0}\colon (V'_{i_0}, 
\Delta_{V'_{i_0}}+c_P(f'_{i_0})^*P)\to 
W$ and then use the connectedness of $(f'|_{V'})^{-1}(P)$. 
By repeating this argument, 
we finally obtain that $B\cap F$ always contains a log canonical 
center of $(V'_i, \Delta_{V'_i}+c_P(f'_i)^*P)$. 
\end{proof}
As we mentioned above, we see that $\alpha_{P, V'_i}$ is 
independent of $i\in I$. 
This is what we wanted. 
\end{step}

The above arguments show that $\mathbf B_W
=\mathbf B(W; X, \Delta)_W$ holds. 
We finish the proof of Theorem \ref{f-thm1.2}. 
\end{proof}

\begin{rem}\label{f-rem3.5}
In Step \ref{f-step3} in the proof of Theorem \ref{f-thm1.2}, 
we proved that $f'\colon V'\to W$ has 
connected fibers. 
Note that $W$ is normal by the reduction argument 
in Step \ref{f-step1} in the proof of Theorem \ref{f-thm1.2}. 
However, it is not clear whether $f'\colon V'_i\to W$ has connected 
fibers or not. 
Hence we need a somewhat artificial formulation in 
Lemma \ref{f-lem3.1}. 
\end{rem}
%%%%%%%%%%%%%%%

\end{document}